\documentclass[11pt]{amsart}

\usepackage[margin=1in]{geometry}
\usepackage[T1]{fontenc}
\usepackage{amsmath,amssymb,amsthm,url,enumitem}
\usepackage{tikz}
\usetikzlibrary{arrows.meta}

\newtheorem{theorem}{Theorem}
\newtheorem{corollary}[theorem]{Corollary}

\newtheorem{proposition}[theorem]{Proposition}
\newtheorem{remark}[theorem]{Remark}

\newcommand{\N}{\mathbb{N}}

\title{Forbidden subgraphs in divisor graphs and an Erd\H{o}s divisibility problem}
\author{Damek Davis}
\thanks{Research supported by NSF DMS award 2523384.}
\address{Department of Statistics and Data Science, The Wharton School, University of Pennsylvania}
\email{damek@wharton.upenn.edu}
\date{}

\begin{document}

\begin{abstract}
Erd\H{o}s asked for the largest size $f(n)$ of a subset of $\{1,\dots,n\}$ with no element dividing two others, and whether $\lim f(n)/n$ is irrational. We show that $f(n)=c_2\,n+o(n)$ for an effectively computable constant~$c_2$, and moreover that the number $q(n)$ of such subsets satisfies $q(n)=\beta_2^{n+o(n)}$ for a computable constant~$\beta_2$. To prove this, we recast the divisibility constraint as forbidding a certain subgraph in the directed divisor graph on $\{1,\dots,n\}$ and prove a more general result: for any finite family of connected forbidden subgraphs of the divisor graph, both the extremal density and counting rate are effectively computable. The proof uses a theorem of McNew on local statistics of divisor graphs. The irrationality of~$c_2$ remains open.
\end{abstract}

\maketitle

\section{Introduction}

Let \(f(n)\) denote the largest size of a set \(A\subseteq \{1,\dots,n\}\) such that there do not exist distinct \(a,b,c\in A\) with \(a\mid b\) and \(a\mid c\). A question of Erd\H{o}s (see Guy~\cite[Problem~B24]{Gu04} and Bloom~\cite[Problem 1062]{BloomEP}) asks how large \(f(n)\) can be and whether \(\lim_{n\to\infty} f(n)/n\) is irrational. The interval \(\{\lfloor n/3\rfloor+1,\dots,n\}\) already shows that \(f(n)\ge \lceil 2n/3\rceil\). Lebensold~\cite{Le76} proved that for sufficiently large \(n\),
\[
0.6725\,n\le f(n)\le 0.6736\,n,
\]
and OEIS~A038372 tabulates the exact extremal values for small \(n\)~\cite{OEISA038372}.

In this note we prove that for every \(\varepsilon > 0\),
\[
f(n)=c_2\,n+O_{\varepsilon}\!\left(n\exp\!\bigl(-(1-\varepsilon)\sqrt{\log n\,\log\log n}\bigr)\right)
\]
for an effectively computable constant \(c_2\), and the number of such subsets has exponential growth rate~\(\beta_2\) for a computable constant~\(\beta_2\ge 1\) (Corollary~\ref{cor:erdos}). The key is to recast the problem in the \emph{divisor graph} on \(\{1,\dots,n\}\), whose vertices are joined whenever one divides the other; divisibility gives each edge a natural orientation from divisor to multiple, written \(u\to v\) when \(u\mid v\). The Erd\H{o}s condition---no element divides two others---is then equivalent to forbidding the directed \emph{two-fork} \(x\to y\), \(x\to z\). This makes the problem an instance of a \emph{vertex Tur\'an-type problem} (e.g.,~\cite{JoTa10}): given a host graph~\(H\) and a forbidden subgraph~\(F\), find the largest vertex subset \(U\subseteq V(H)\) whose induced subgraph contains no copy of~\(F\). (The classical theorem of Tur\'an~\cite{Tu41} is the edge-maximization analogue.)

Our main result (Theorem~\ref{thm:general}) shows that for any finite family of connected forbidden subgraphs of the divisor graph, the largest avoiding subset of \(\{1,\dots,n\}\) has size \(c\,n+o(n)\) for an effectively computable constant~\(c\) (Corollary~\ref{cor:patterns}). In addition, the number of avoiding subsets has exponential growth rate~\(\beta\) for an effectively computable constant~\(\beta\ge 1\). The forbidden subgraphs may be directed or undirected; for the Erd\H{o}s problem the orientation is essential, since the undirected two-fork would also forbid the reverse pattern---no element is a common multiple of two others---which is a different condition. The Erd\H{o}s two-fork is one instance (Corollary~\ref{cor:erdos}); further applications to \(r\)-fork avoidance and \(k\)-chain avoidance are given in Remark~\ref{par:applications}.

The main tool is a theorem of McNew~\cite{McNew} on local statistics of divisor graphs, which expresses certain additive and multiplicative divisor-graph statistics as explicit convergent series. Theorem~\ref{thm:general} identifies a broad class of properties---those that are downward closed, decompose over coprime components, and are preserved by integer dilation---for which we use McNew's theorem to obtain both computable density and counting estimates. For the Erd\H{o}s two-fork case, the numerical estimates in Section~\ref{sec:numerical} improve Lebensold's lower bound to \(c_2\ge 0.6729\ldots\), narrowing the gap with his upper bound \(c_2\le 0.6736\) to approximately \(0.0007\). We also estimate \(1.729\ldots \le \beta_2 \le 1.874\ldots\).

McNew~\cite{McNew} also uses his theorem to study other problems on the divisor graph. The most-studied case is the two-chain \(a\mid b\): sets avoiding it are called \emph{primitive}, and their extremal density is \(1/2\). Determining the counting rate is more involved: Angelo~\cite{Angelo18} proved that the number of primitive subsets of \(\{1,\dots,n\}\) has an exponential growth rate, and McNew~\cite[Theorem~2]{McNew} obtained this rate with an explicit error term.

\section{Main results}\label{sec:general}

This section states the general theorem (Theorem~\ref{thm:general}), which gives both an extremal density and a counting result, derives its forbidden-subgraph corollary (Corollary~\ref{cor:patterns}), and specializes to the Erd\H{o}s problem (Corollary~\ref{cor:erdos}). We begin with the notation needed for the constants that appear in these results.

For any family \(\mathcal{P}\) of finite subsets of \(\N\) and any finite set \(T\subset\N\), write
\[
\Phi_{\mathcal{P}}(T):=\max\{|B|:B\subseteq T,\;B\in\mathcal{P}\},
\qquad
Q_{\mathcal{P}}(T):=\#\{B\subseteq T:B\in\mathcal{P}\}.
\]
Let \(P^+(d)\) denote the largest prime factor of \(d\), with \(P^+(1)=1\). For any function \(\Psi\) defined on finite subsets of \(\N\) for which the following series converges absolutely, define
\[
C_{\Psi}:=
\sum_{i=1}^{\infty}
\left(\prod_{p\le i}\frac{p-1}{p}\right)
\sum_{\substack{d\ge 1\\ P^+(d)\le i}}
\sum_{t\in[id,(i+1)d)}
\frac{\Psi(\{d,\dots,t\})-\Psi(\{d+1,\dots,t\})}{t(t+1)}.
\]
The coefficient weights in this series sum to~\(1\) (as we later see in~\eqref{eq:weight}), so \(C_\Psi\) converges absolutely whenever the differences \(\Psi(\{d,\dots,t\})-\Psi(\{d+1,\dots,t\})\) are uniformly bounded. In particular, \(C_{\Phi_{\mathcal{P}}}\) and \(C_{\log Q_{\mathcal{P}}}\) are well-defined for any family \(\mathcal{P}\) satisfying the hypotheses of the theorem below.

\begin{theorem}\label{thm:general}
Let \(\mathcal{P}\) be a family of finite subsets of \(\N\), called admissible sets, and assume \(\emptyset\in\mathcal{P}\). Assume further:
\begin{enumerate}
\item (downward closed) if \(B\in \mathcal{P}\) and \(B'\subseteq B\), then \(B'\in \mathcal{P}\);
\item (decomposes over coprime components) if \(T=T_1\sqcup T_2\) is finite and there is no divisibility relation between an element of \(T_1\) and an element of \(T_2\), then for every \(B\subseteq T\) we have \(B\in \mathcal{P}\) if and only if \(B\cap T_i\in \mathcal{P}\) for \(i=1,2\);
\item (preserved by integer dilation) if \(m\in \N\) and \(B\) is finite, then \(B\in \mathcal{P}\) if and only if \(mB\in \mathcal{P}\), where \(mB:=\{mb:\ b\in B\}\).
\end{enumerate}
Write \(f_{\mathcal{P}}(n):=\Phi_{\mathcal{P}}(\{1,\dots,n\})\), \(q_{\mathcal{P}}(n):=Q_{\mathcal{P}}(\{1,\dots,n\})\), \(c_{\mathcal{P}}:=C_{\Phi_{\mathcal{P}}}\), and \(\beta_{\mathcal{P}}:=\exp(C_{\log Q_{\mathcal{P}}})\).
Then for every \(\varepsilon>0\):
\begin{enumerate}[label=\textup{(\alph*)}]
\item \textup{(Extremal density.)}
\[
f_{\mathcal{P}}(n)=c_{\mathcal{P}}\,n+O_{\varepsilon}\!\left(
n\exp\!\bigl(-(1-\varepsilon)\sqrt{\log n\,\log\log n}\bigr)
\right).
\]

\item \textup{(Counting.)}
\[
\log q_{\mathcal{P}}(n)=n\log\beta_{\mathcal{P}}+O_{\varepsilon}\!\left(
n\exp\!\bigl(-(1-\varepsilon)\sqrt{\log n\,\log\log n}\bigr)
\right).
\]
In particular, \(\lim_{n\to\infty}q_{\mathcal{P}}(n)^{1/n}=\beta_{\mathcal{P}}\).
\end{enumerate}
If in addition membership in \(\mathcal{P}\) is decidable for finite sets, then \(c_{\mathcal{P}}\) and \(\beta_{\mathcal{P}}\) are effectively computable.
\end{theorem}

\begin{remark}[Partition function]\label{rem:partition}\rm
The same argument shows that for any fixed \(z>0\), the partition function \(Z_{\mathcal{P}}(T,z):=\sum_{B\subseteq T,\,B\in\mathcal{P}} z^{|B|}\) satisfies
\[
\log Z_{\mathcal{P}}(\{1,\dots,n\},z)
=\kappa_{\mathcal{P}}(\log z)\,n
+O_{\varepsilon}\!\left(
\log(1+z)\,n\exp\!\bigl(-(1-\varepsilon)\sqrt{\log n\,\log\log n}\bigr)
\right)
\]
for an effectively computable constant \(\kappa_{\mathcal{P}}(\log z)\). The case \(z=1\) is part~(b) of the theorem. It is straightforward to show that the ``pressure" \(t\mapsto\kappa_{\mathcal{P}}(t)\) is convex and \(1\)-Lipschitz~\cite{CGT08}, and \(\kappa_{\mathcal{P}}(t)/t\to c_{\mathcal{P}}\) as \(t\to\infty\).
\end{remark}

We now describe the family of admissible sets to which we apply Theorem~\ref{thm:general}. The \emph{divisor graph} of a finite set \(S\subset\N\) has vertex set \(S\) with an edge between \(u\) and \(v\) whenever \(u\ne v\) and one divides the other; we orient each edge from divisor to multiple, writing \(u\to v\) when \(u\mid v\). Given a collection \(\mathcal{F}\) of connected graphs, each directed or undirected---here connected refers to the underlying undirected graph, regardless of orientation---say that a finite set \(S\) is \emph{\(\mathcal{F}\)-free} if its divisor graph contains no copy of any member of~\(\mathcal{F}\): containment of a directed member means containment as a directed subgraph of the oriented divisor graph, while containment of an undirected member means containment as a subgraph of the underlying undirected graph. Write \(\mathcal{P}(\mathcal{F})\) for the family of all \(\mathcal{F}\)-free sets. We do not require the copy to be induced, so extra divisibility arrows among the chosen vertices are allowed.\footnote{If one instead requires induced containment, the three axioms still hold and Theorem~\ref{thm:general} applies equally. Non-induced avoidance is more restrictive and is the natural choice for the Erd\H{o}s problem, since ``no element divides two others'' does not depend on the relationship between those two others.} For example, \(\{1,2,4\}\) contains the two-fork \(x\to y\), \(x\to z\): the arrows \(1\to 2\) and \(1\to 4\) realize it, while the extra arrow \(2\to 4\) is harmless. Figure~\ref{fig:noninduced} illustrates this distinction, and Figure~\ref{fig:patterns} shows four representative connected forbidden subgraphs. The family \(\mathcal{P}(\mathcal{F})\) satisfies the three axioms of Theorem~\ref{thm:general} (proof in Section~\ref{sec:patterns}), giving the following corollary.

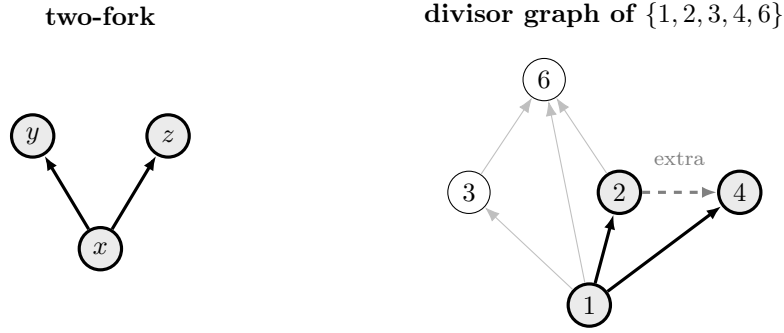
\begin{figure}[t]
\centering
\begin{tikzpicture}[
    vertex/.style={circle,draw,inner sep=1.2pt,minimum size=16pt,font=\small},
    matched/.style={circle,draw,line width=1.2pt,fill=black!8,inner sep=1.2pt,minimum size=16pt,font=\small},
    bg/.style={-{Latex[length=2mm]},thin,black!25},
    pat/.style={-{Latex[length=2mm]},line width=1.2pt},
    extra/.style={-{Latex[length=2mm]},line width=1.2pt,dashed,black!50}
]
%%% Left: two-fork pattern (vertically centered with right graph)
\node[font=\small\bfseries] at (-2.5,3.85) {two-fork};
\begin{scope}[shift={(-2.5,0.75)}]
\node[matched] (px) at (0,0) {$x$};
\node[matched] (py) at (-0.9,1.5) {$y$};
\node[matched] (pz) at (0.9,1.5) {$z$};
\draw[pat] (px) -- (py);
\draw[pat] (px) -- (pz);
\end{scope}

%%% Right: divisor graph of {1,2,3,4,6}
\begin{scope}[shift={(4.0,0)}]
\node[font=\small\bfseries] at (0.2,3.85) {divisor graph of $\{1,2,3,4,6\}$};
% level 0: root
\node[matched] (v1) at (0,0) {$1$};
% level 1: primes/prime-powers
\node[vertex]  (v3) at (-1.6,1.5) {$3$};
\node[matched] (v2) at (0.4,1.5) {$2$};
\node[matched] (v4) at (2.0,1.5) {$4$};
% level 2: composite — centered between its divisors 2 and 3
\node[vertex]  (v6) at (-0.6,3.0) {$6$};
% background divisibility arrows (faded)
\draw[bg] (v1) -- (v3);
\draw[bg] (v1) -- (v6);
\draw[bg] (v3) -- (v6);
\draw[bg] (v2) -- (v6);
% pattern-matching arrows (bold): 1->2 and 1->4
\draw[pat] (v1) -- (v2);
\draw[pat] (v1) -- (v4);
% extra divisibility among matched vertices (dashed): 2|4
\draw[extra] (v2) -- (v4);
\node[font=\scriptsize,black!50] at (1.2,1.95) {extra};
\end{scope}
\end{tikzpicture}
\caption{Non-induced subgraph containment. \textbf{Left:} the two-fork. \textbf{Right:} the divisor graph of $\{1,2,3,4,6\}$; the shaded vertices $\{1,2,4\}$ realize the two-fork via the bold arrows $1\to 2$ and $1\to 4$. The dashed arrow $2\to 4$ is an extra divisibility among the matched vertices; the copy need not be induced.}
\label{fig:noninduced}
\end{figure}

\begin{figure}[t]
\centering
\begin{tikzpicture}[
    nd/.style={circle,draw,line width=1.2pt,fill=black!8,inner sep=1.2pt,minimum size=16pt,font=\small},
    arr/.style={-{Latex[length=2mm]},line width=1.2pt}
]
% two-fork
\begin{scope}[shift={(0,0)}]
\node[font=\small\bfseries] at (0,2.4) {two-fork};
\node[nd] (a1) at (0,0) {$x$};
\node[nd] (b1) at (-0.9,1.5) {$y$};
\node[nd] (c1) at (0.9,1.5) {$z$};
\draw[arr] (a1) -- (b1);
\draw[arr] (a1) -- (c1);
\end{scope}

% three-fork
\begin{scope}[shift={(5,0)}]
\node[font=\small\bfseries] at (0,2.4) {three-fork};
\node[nd] (a2) at (0,0) {$x$};
\node[nd] (b2) at (-1.1,1.5) {$y$};
\node[nd] (c2) at (0,1.5) {$z$};
\node[nd] (d2) at (1.1,1.5) {$w$};
\draw[arr] (a2) -- (b2);
\draw[arr] (a2) -- (c2);
\draw[arr] (a2) -- (d2);
\end{scope}

% chain
\begin{scope}[shift={(0,-3.5)}]
\node[font=\small\bfseries] at (0,2.4) {chain};
\node[nd] (a3) at (-2.0,1.0) {$x_0$};
\node[nd] (b3) at (-0.65,1.0) {$x_1$};
\node[nd] (c3) at (0.65,1.0) {$x_2$};
\node[nd] (d3) at (2.0,1.0) {$x_3$};
\draw[arr] (a3) -- (b3);
\draw[arr] (b3) -- (c3);
\draw[arr] (c3) -- (d3);
\end{scope}

% in-fork
\begin{scope}[shift={(5,-3.5)}]
\node[font=\small\bfseries] at (0,2.4) {in-fork};
\node[nd] (a4) at (0,1.5) {$x$};
\node[nd] (b4) at (-0.9,0) {$y$};
\node[nd] (c4) at (0.9,0) {$z$};
\draw[arr] (b4) -- (a4);
\draw[arr] (c4) -- (a4);
\end{scope}
\end{tikzpicture}
\caption{Representative connected forbidden subgraphs of divisor graphs. Arrows point from divisor to multiple.}
\label{fig:patterns}
\end{figure}
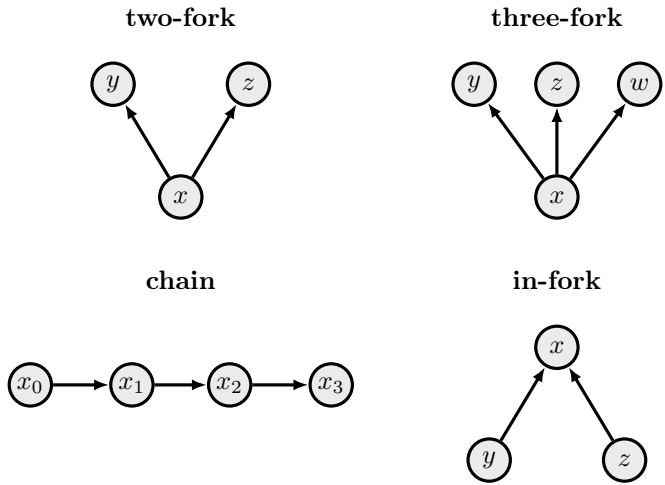

\begin{corollary}\label{cor:patterns}
Let \(\mathcal{F}\) be a finite family of connected\footnote{Connectedness is essential: a disconnected forbidden subgraph can straddle two connected components of the divisor graph, so the componentwise axiom may fail.} forbidden subgraphs (directed or undirected) of divisor graphs, and let \(f_{\mathcal{F}}(n)\) and \(q_{\mathcal{F}}(n)\) denote the maximum size, and the number, of \(\mathcal{F}\)-free subsets of \(\{1,\dots,n\}\). Applying Theorem~\ref{thm:general} to \(\mathcal{P}(\mathcal{F})\), there exist effectively computable constants \(c_{\mathcal{F}}\) and \(\beta_{\mathcal{F}}\ge 1\) such that, for every \(\varepsilon>0\),
\begin{align*}
f_{\mathcal{F}}(n)&=c_{\mathcal{F}}\,n+O_{\varepsilon}\!\bigl(n\exp\!\bigl(-(1-\varepsilon)\sqrt{\log n\,\log\log n}\bigr)\bigr),\\
\log q_{\mathcal{F}}(n)&=n\log\beta_{\mathcal{F}}+O_{\varepsilon}\!\bigl(n\exp\!\bigl(-(1-\varepsilon)\sqrt{\log n\,\log\log n}\bigr)\bigr).
\end{align*}
\end{corollary}

Corollary~\ref{cor:patterns} is not the only source of examples for Theorem~\ref{thm:general}. The collection \(\mathcal{F}\) may also be infinite, or may mix directed and undirected patterns; the only requirement for Theorem~\ref{thm:general} is that the resulting family \(\mathcal{P}(\mathcal{F})\) be decidable and satisfy the three axioms. (Finiteness of~\(\mathcal{F}\) in Corollary~\ref{cor:patterns} guarantees decidability automatically.) For instance, requiring the divisor graph of the chosen subset to be a forest---equivalently, forbidding all cycles as undirected subgraphs---satisfies the three axioms, so Theorem~\ref{thm:general} applies. This is a genuinely infinite family of forbidden subgraphs: for every \(k\ge 3\), taking \(k\)~distinct primes \(p_1,\dots,p_k\) and the products \(p_1 p_2,\, p_2 p_3,\,\dots,\, p_k p_1\) yields a set of~\(2k\) integers whose divisor graph is exactly~\(C_{2k}\) with no shorter cycle, so no finite subfamily suffices.

Specializing the corollary to the two-fork \(x\mid y\), \(x\mid z\) recovers the Erd\H{o}s problem:

\begin{corollary}\label{cor:erdos}
Let \(f(n)\) and \(q(n)\) denote the maximum size, and the number, of subsets of \(\{1,\dots,n\}\) containing no distinct \(x,y,z\) with \(x\mid y\) and \(x\mid z\). Then there exist effectively computable constants \(c_2\) and \(\beta_2\ge 1\) such that, for every \(\varepsilon>0\),
\begin{align*}
f(n)&=c_2\,n+O_{\varepsilon}\!\left(n\exp\!\bigl(-(1-\varepsilon)\sqrt{\log n\,\log\log n}\bigr)\right),\\
\log q(n)&=n\log\beta_2+O_{\varepsilon}\!\left(n\exp\!\bigl(-(1-\varepsilon)\sqrt{\log n\,\log\log n}\bigr)\right).
\end{align*}
In particular, \(\lim_{n\to\infty}q(n)^{1/n}=\beta_2\).
\end{corollary}

The question of whether \(c_2\) is irrational, raised by Erd\H{o}s (see~\cite[Problem~B24]{Gu04}), remains open.

\begin{remark}[Forks, in-forks, and chains]\label{par:applications}\rm
Corollary~\ref{cor:patterns} applies equally to other connected forbidden subgraphs. For each \(r\ge 2\), the \(r\)-fork---forbidding any element from dividing \(r\) others---is connected, so the corollary gives computable density \(c_r\) and counting rate. The interval \((\lfloor n/(r+1)\rfloor,n]\) gives \(c_r\ge r/(r+1)\) and, since all of its subsets are admissible, \(\beta_r\ge 2^{r/(r+1)}\). For \(r\ge 3\), earlier bounds on \(c_r\) were obtained by Vijay~\cite{Vijay06} and Hegarty~\cite{Hegarty06}.

For each \(k\ge 2\), the directed path \(x_1\to x_2\to\cdots\to x_k\) on \(k\)~vertices is connected, so the corollary gives a computable density \(\gamma_k\) and counting rate \(\delta_k\) for \(k\)-chain-free subsets of \(\{1,\dots,n\}\). The interval \((\lfloor n/2^{k-1}\rfloor,n]\) is \(k\)-chain-free, giving \(\gamma_k\ge 1-2^{1-k}\) and \(\delta_k\ge 2^{1-2^{1-k}}\). For \(k=2\), \(\gamma_2=1/2\): a \(2\)-chain-free set is an antichain, and the \(\lceil n/2\rceil\) chains \(\{q,2q,4q,\dots\}\cap[1,n]\) for odd \(q\le n\) show this is tight.

Dually, the \(r\)-in-fork \(y_1\to x,\,\dots,\,y_r\to x\) forbids any element from being a multiple of \(r\) others. This is the problem dual to the \(r\)-fork, also posed by Erd\H{o}s~\cite[Problem~B24]{Gu04}. The in-fork is connected and directed, so Corollary~\ref{cor:patterns} gives a computable density \(c_r^*\) and counting rate for each \(r\ge 2\). The same interval \((\lfloor n/(r+1)\rfloor,n]\) is also \(r\)-in-fork-free, giving \(c_r^*\ge r/(r+1)\). Following the same techniques as in Section~\ref{sec:numerical}, we obtain \(0.7195\le c_2^*\le 0.788\) and \(1.82\le\beta_2^*\le 1.91\), confirming that the in-fork and two-fork problems have different extremal densities.
\end{remark}

\section{Proof of Theorem \ref{thm:general}}

We deduce Theorem~\ref{thm:general} from the following result of McNew. For integers \(1\le a\le n\), let \(D[a,n]\) be the divisor graph on the vertex set \(\{a,a+1,\dots,n\}\), with an edge between \(u\) and \(v\) when \(u\mid v\) or \(v\mid u\).

\begin{theorem}[McNew {\cite[Theorem~3 and footnote~1]{McNew}}]\label{thm:mcnew}
Fix \(\varepsilon>0\). Let \(h(a,n)\) be defined for \(1\le a\le n\), assume \(|h(a,n)|\le M\), and suppose that \(h(a,n)\) depends only on the rooted connected component of \(a\) in \(D[a,n]\), that is, on the connected component together with the distinguished vertex \(a\), up to linear scaling isomorphism.\footnote{McNew states the theorem for rooted graph isomorphism; the accompanying footnote observes that one may restrict the definition to linear isomorphisms without affecting the proof.}
Then there exists a constant \(c_h\) such that
\[
\sum_{a=1}^{n} h(a,n)=c_h\,n+O_{\varepsilon}\!\left(
M n\exp\!\bigl(-(1-\varepsilon)\sqrt{\log n\,\log\log n}\bigr)
\right).
\]
Moreover, if \(P^+(d)\) denotes the largest prime factor of \(d\), with the convention \(P^+(1)=1\), then
\[
c_h=
\sum_{i=1}^{\infty}
\left(\prod_{p\le i}\frac{p-1}{p}\right)
\sum_{\substack{d\ge 1\\ P^+(d)\le i}}
\sum_{t\in[id,(i+1)d)}
\frac{h(d,t)}{t(t+1)}.
\]
\end{theorem}

The proof verifies McNew's hypotheses for suitable local functions and then recovers the stated asymptotics by telescoping.

\begin{proof}[Proof of Theorem~\ref{thm:general}]
\noindent\textit{Part~\textup{(a)}: extremal density.}
Write \(F_{\mathcal{P}}(a,n):=\Phi_{\mathcal{P}}(\{a,\dots,n\})\) and \(g_{\mathcal{P}}(a,n):=F_{\mathcal{P}}(a,n)-F_{\mathcal{P}}(a+1,n)\), with \(F_{\mathcal{P}}(n+1,n)=0\).

Since \(\{a+1,\dots,n\}\subseteq \{a,\dots,n\}\), we have \(F_{\mathcal{P}}(a,n)\ge F_{\mathcal{P}}(a+1,n)\). By downward closure, if \(B\subseteq \{a,\dots,n\}\) is admissible, then \(B\setminus\{a\}\subseteq \{a+1,\dots,n\}\) is admissible, so \(F_{\mathcal{P}}(a+1,n)\ge F_{\mathcal{P}}(a,n)-1\).
Hence \(0\le g_{\mathcal{P}}(a,n)\le 1\).

Next we show that \(g_{\mathcal{P}}(a,n)\) depends only on the rooted connected component of \(a\) in \(D[a,n]\). If \(T=T_1\sqcup T_2\) and there is no divisibility relation between an element of \(T_1\) and an element of \(T_2\), then the second axiom says that a subset \(B\subseteq T\) is admissible if and only if \(B\cap T_i\in \mathcal{P}\) for \(i=1,2\). Thus admissible subsets of \(T\) are exactly the unions \(B_1\sqcup B_2\) with \(B_i\in \mathcal{P}\), so \(\Phi_{\mathcal{P}}(T)=\Phi_{\mathcal{P}}(T_1)+\Phi_{\mathcal{P}}(T_2)\).
Now let \(C(a,n)\) be the connected component of \(a\) in \(D[a,n]\), and let
\[
R(a,n):=\{a,\dots,n\}\setminus C(a,n).
\]
Note that \(\{a,\dots,n\}=C(a,n)\sqcup R(a,n)\) and \(\{a+1,\dots,n\}=(C(a,n)\setminus\{a\})\sqcup R(a,n)\), with no divisor-graph edges between \(C(a,n)\) and \(R(a,n)\). Hence
\[
F_{\mathcal{P}}(a,n)=\Phi_{\mathcal{P}}(C(a,n))+\Phi_{\mathcal{P}}(R(a,n)),
\qquad
F_{\mathcal{P}}(a+1,n)=\Phi_{\mathcal{P}}(C(a,n)\setminus\{a\})+\Phi_{\mathcal{P}}(R(a,n)),
\]
and subtraction gives
\[
g_{\mathcal{P}}(a,n)
=
\Phi_{\mathcal{P}}(C(a,n))
-
\Phi_{\mathcal{P}}(C(a,n)\setminus\{a\}).
\]
Thus \(g_{\mathcal{P}}(a,n)\) depends only on the rooted connected component of \(a\) in \(D[a,n]\).

If \(C(a,n)\) and \(C(b,m)\) are linearly isomorphic, choose \(\lambda=u/v>0\) with \(u,v\in \N\) and \(u\,C(a,n)=v\,C(b,m)\). By scale invariance,
\[
\Phi_{\mathcal{P}}(C(a,n))
=
\Phi_{\mathcal{P}}(C(b,m)),
\qquad
\Phi_{\mathcal{P}}(C(a,n)\setminus\{a\})
=
\Phi_{\mathcal{P}}(C(b,m)\setminus\{b\}).
\]
Hence \(g_{\mathcal{P}}(a,n)=g_{\mathcal{P}}(b,m)\). Therefore \(g_{\mathcal{P}}\) satisfies Theorem~\ref{thm:mcnew} with \(M=1\), and so
\[
\sum_{a=1}^{n} g_{\mathcal{P}}(a,n)
=
c_{\mathcal{P}}\,n
+
O_{\varepsilon}\!\left(
n\exp\!\bigl(-(1-\varepsilon)\sqrt{\log n\,\log\log n}\bigr)
\right),
\]
Comparing McNew's formula for \(c_h\) with the definition of \(C_\Psi\) shows that \(c_{\mathcal{P}}=C_{\Phi_{\mathcal{P}}}\), as claimed. Telescoping gives
\[
f_{\mathcal{P}}(n)
=
F_{\mathcal{P}}(1,n)
=
\sum_{a=1}^{n}\bigl(F_{\mathcal{P}}(a,n)-F_{\mathcal{P}}(a+1,n)\bigr)
=
\sum_{a=1}^{n} g_{\mathcal{P}}(a,n),
\]
which yields the stated asymptotic.

It remains to show effective computability. Since \(0\le g_{\mathcal{P}}(d,t)\le 1\), every term in the series \(C_{\Phi_{\mathcal{P}}}\) is nonnegative. Write \(w_{i,d,t}:=\bigl(\prod_{p\le i}\tfrac{p-1}{p}\bigr)\tfrac{1}{t(t+1)}\) for the coefficient weight of the \((i,d,t)\) term, so that \(C_{\Phi_{\mathcal{P}}}=\sum w_{i,d,t}\,g_{\mathcal{P}}(d,t)\). By partial fractions and the Euler product for smooth numbers,
\begin{equation}\label{eq:weight}
\sum_{t\in[id,(i+1)d)} \frac{1}{t(t+1)}=\frac{1}{i(i+1)d}
\qquad\text{and}\qquad
\sum_{\substack{d\ge 1\\ P^+(d)\le i}} \frac{1}{d}=\prod_{p\le i}\!\left(1-\frac{1}{p}\right)^{-1}\!,
\end{equation}
so the total weight is \(\sum w_{i,d,t}=\sum_{i=1}^{\infty}1/(i(i+1))=1\). In particular, the series \(C_{\Phi_{\mathcal{P}}}\) converges absolutely. If membership in \(\mathcal{P}\) is decidable for finite sets, then \(g_{\mathcal{P}}(d,t)\) is computable by exhaustive search. Enumerate the triples \((i,d,t)\) in any computable order, and let
\[
S_N:=\sum_{n\le N} w_{i_n,d_n,t_n}\,g_{\mathcal{P}}(d_n,t_n),
\qquad
W_N:=\sum_{n\le N} w_{i_n,d_n,t_n}.
\]
Since all terms are nonnegative and \(\sum w_{i,d,t}=1\), we have
\[
0\le c_{\mathcal{P}}-S_N\le 1-W_N.
\]
Thus once \(1-W_N<\delta\), the partial sum \(S_N\) approximates \(c_{\mathcal{P}}\) to within~\(\delta\).

\medskip\noindent\textit{Part~\textup{(b)}: counting.}
The argument is the multiplicative analogue of part~(a). Write \(Q_{\mathcal{P}}(a,n):=Q_{\mathcal{P}}(\{a,\dots,n\})\), with \(Q_{\mathcal{P}}(n+1,n)=1\). If \(T=T_1\sqcup T_2\) with no divisibility between \(T_1\) and \(T_2\), the componentwise axiom says \(B\subseteq T\) is admissible if and only if \(B\cap T_i\) is admissible for \(i=1,2\), so admissible subsets of \(T\) are exactly the unions \(B_1\sqcup B_2\) with each \(B_i\subseteq T_i\) admissible. Hence \(Q_{\mathcal{P}}(T)=Q_{\mathcal{P}}(T_1)\,Q_{\mathcal{P}}(T_2)\). Define
\[
h_{\mathcal{P}}(a,n):=\log\frac{Q_{\mathcal{P}}(a,n)}{Q_{\mathcal{P}}(a+1,n)}.
\]
The multiplicative decomposition over \(C(a,n)\sqcup R(a,n)\) cancels \(Q_{\mathcal{P}}(R(a,n))\) in numerator and denominator, so \(h_{\mathcal{P}}(a,n)\) depends only on the rooted component of~\(a\). Every admissible subset of \(\{a+1,\dots,n\}\) is also an admissible subset of \(\{a,\dots,n\}\), so \(Q_{\mathcal{P}}(a,n)\ge Q_{\mathcal{P}}(a+1,n)\). Conversely, every admissible \(B\subseteq\{a,\dots,n\}\) either omits \(a\) or contains it, and removing \(a\) preserves admissibility by downward closure, so \(Q_{\mathcal{P}}(a,n)\le 2\,Q_{\mathcal{P}}(a+1,n)\). Hence \(0\le h_{\mathcal{P}}(a,n)\le\log 2\). Scale invariance preserves both counts, so \(h_{\mathcal{P}}\) is invariant under linear isomorphism of rooted components. Hence \(h_{\mathcal{P}}\) satisfies Theorem~\ref{thm:mcnew} with \(M=\log 2\). The ratios telescope: \(\log q_{\mathcal{P}}(n)=\sum_{a=1}^n h_{\mathcal{P}}(a,n)\), and McNew gives \(\sum h_{\mathcal{P}}(a,n)=C_{\log Q_{\mathcal{P}}}\,n+O_\varepsilon(\cdots)\), so \(\beta_{\mathcal{P}}=\exp(C_{\log Q_{\mathcal{P}}})\). Effective computability follows from the same weight argument as in part (a), with \(\log 2\) replacing~\(1\).
\end{proof}

\section{Connected forbidden subgraphs}\label{sec:patterns}

This section verifies the hypotheses of Theorem~\ref{thm:general} for forbidden-subgraph avoidance and proves Corollary~\ref{cor:patterns}.

Fix a finite family \(\mathcal{F}\) of connected forbidden subgraphs (directed or undirected). Let \(\mathcal{P}(\mathcal{F})\) be the family of finite sets \(B\subset \N\) whose divisor graph contains no copy of any member of~\(\mathcal{F}\).
\begin{proposition}\label{prop:patterns}
The family \(\mathcal{P}(\mathcal{F})\) satisfies the hypotheses of Theorem~\ref{thm:general}.
\end{proposition}

\begin{proof}
Clearly \(\emptyset\in \mathcal{P}(\mathcal{F})\). The downward-closed axiom is immediate: passing to a subset cannot create a forbidden subgraph. For scale invariance, note that the axiom only asks about integer dilations. If \(m\in\N\), then multiplication by \(m\) preserves divisibility, so a finite set \(B\) contains a forbidden subgraph if and only if \(mB\) does. Hence \(B\in\mathcal{P}(\mathcal{F})\) if and only if \(mB\in\mathcal{P}(\mathcal{F})\).

For the componentwise axiom, let \(T=T_1\sqcup T_2\) with no divisibility relation between \(T_1\) and \(T_2\). If \(B\subseteq T\) contains a copy of some forbidden subgraph in \(\mathcal{F}\), let \(S\subseteq B\) be the set of vertices of that copy. Since the forbidden subgraph is connected after forgetting directions, its image on \(S\) is connected in the undirected divisor graph, so \(S\) lies entirely in \(T_1\) or entirely in \(T_2\). Therefore \(B\in\mathcal{P}(\mathcal{F})\) if and only if both \(B\cap T_1\) and \(B\cap T_2\) lie in \(\mathcal{P}(\mathcal{F})\).
\end{proof}

\begin{proof}[Proof of Corollary~\ref{cor:patterns}]
Apply Theorem~\ref{thm:general} with \(\mathcal{P}=\mathcal{P}(\mathcal{F})\), using Proposition~\ref{prop:patterns}. Effective computability follows because \(\mathcal{F}\) is finite, so membership in \(\mathcal{P}(\mathcal{F})\) is decidable for finite sets.
\end{proof}

\section{Numerical estimates}\label{sec:numerical}

Both constants \(c_2\) and \(\beta_2\) are computed from the same series \(C_\Psi\), with \(\Psi=\Phi\) for the density and \(\Psi=\log Q\) for the counting rate. Write \(g(d,t):=\Psi(\{d,\dots,t\})-\Psi(\{d+1,\dots,t\})\) for the local increment. In both cases \(g(d,t)\ge 0\), so every term in the series is nonnegative and any finite partial sum is a lower bound on~\(C_\Psi\). The coefficient weights \(w_{i,d,t}:=\bigl(\prod_{p\le i}\tfrac{p-1}{p}\bigr)\tfrac{1}{t(t+1)}\) sum to~\(1\) (see~\eqref{eq:weight}), and the maximum contribution of a fixed \((i,d)\) block is
\[
w_{i,d}\;:=\;\sum_{t\in[id,(i+1)d)} w_{i,d,t}\;=\;\frac{1}{i(i+1)d}\prod_{p\le i}\frac{p-1}{p},
\]
so small \(d\) and especially small \(i\) carry most of the mass. We truncate to triples satisfying \(d\,i^{\alpha}\le B\) for parameters \(\alpha\ge 1\) and \(B\ge 1\). Let \(S\) be the resulting partial sum and \(W=\sum w_{i,d,t}\) the retained coefficient mass. Since \(g(d,t)\le M\) for a known bound~\(M\) (with \(M=1\) for \(c_2\) and \(M=\log 2\) for \(\beta_2\)), the omitted tail is at most \(M(1-W)\), giving
\[
S\;\le\; C_\Psi\;\le\; S + M(1-W).
\]
Both a lower and an upper bound thus follow from a single truncation. Each local increment \(g(d,t)\) is computed exactly for each rooted-component block: for \(c_2\), this requires solving a finite \(0\text{-}1\) optimization; for \(\beta_2\), it requires counting all admissible subsets by dynamic programming.

\subsection{A numerical estimate for \(c_2\)}

For the density, \(g_\Phi(d,t):=\Phi(\{d,\dots,t\})-\Phi(\{d+1,\dots,t\})\) with \(M=1\). We evaluated \(g_\Phi(d,t)\) exactly using the CP-SAT solver in Google OR-Tools. With \(\alpha=10\) and \(B=10^{13}\), the truncation retains \(32660\) rooted-component blocks and gives
\[
c_2\;\ge\; 0.6729.
\]
Combined with Lebensold's upper bound \(c_2\le 0.6736\) from \cite{Le76}, this leaves a gap of approximately \(7\times 10^{-4}\). An analogous upper bound from our series would require the omitted tail to be smaller than Lebensold's bound, but the tail decays as \(1/(I+1)\) for an \(i\le I\) truncation, so reaching this level would require evaluating far more blocks.

\subsection{A numerical estimate for \(\beta_2\)}\label{sec:beta2-numerical}

For the counting rate, \(h_Q(d,t):=\log Q(\{d,\dots,t\})-\log Q(\{d+1,\dots,t\})\) with \(M=\log 2\). We evaluated \(h_Q(d,t)\) exactly by enumerating all subsets of each rooted-component block via dynamic programming, counting how many are two-fork-free for both the block and its root-deleted version. With \(\alpha=10\) and \(B=10^{10}\), the truncation retains \(1143\) blocks. The resulting bounds are
\[
1.729\ldots\le\beta_2\le 1.874\ldots.
\]

\subsection*{Acknowledgments}
The initial version of Corollary~\ref{cor:erdos}, covering only the two-fork (divisor-of-two) case, was proved by ChatGPT~5.4~Pro. The general framework (Theorem~\ref{thm:general} and Corollary~\ref{cor:patterns}) was proposed by the author.

\bibliographystyle{amsalpha}
\bibliography{references}

\end{document}